\documentclass[envcountsame]{llncs}
\usepackage{llncsdoc}
\usepackage{cite}
\usepackage{amsfonts}
\usepackage{amsmath}
\usepackage{amssymb}
\usepackage{color}
\definecolor{darkgreen}{rgb}{0,0.45,0}
\definecolor{darkred}{rgb}{0.9,0,0}
\definecolor{darkblue}{rgb}{0,0,0.6}
\usepackage[colorlinks,citecolor=darkgreen,linkcolor=darkred,urlcolor=darkblue]{hyperref}
\usepackage{breakurl}
\usepackage{enumerate} % for customising enumerated lists

\usepackage[all]{xy}
\xyoption{2cell}
\xyoption{curve}
\UseTwocells
\input{diagxy}

\usepackage{mathpartir}
\usepackage{prooftree}

\newtheorem{examples}[theorem]{Examples}
\newtheorem{defprop}[theorem]{Definition/Proposition}

\newcommand{\Nat}{\mathsf{Nat}}

\newcommand{\C}{\mathbf{C}}
\newcommand{\cC}{\mathcal{C}}

\newcommand{\D}{\mathbb{D}}

\newcommand{\F}{\mathcal{F}}

\newcommand{\cW}{\mathcal{W}}
\newcommand{\cF}{\mathcal{F}}

\newcommand{\Cat}{\mathbf{Cat}}
\newcommand{\Gpd}{\mathbf{Gpd}}

\renewcommand{\lim}{\varprojlim}
\newcommand{\ob}{\operatorname{ob}}
\newcommand{\mor}{\operatorname{mor}}
\newcommand{\iso}{\operatorname{iso}}

\newcommand{\Coeq}{\operatorname{Coeq}}
\newcommand{\op}{\mathrm{op}}

\newcommand{\Sets}{\mathbf{Sets}}
\newcommand{\sSets}{\mathbf{sSets}}

\newcommand{\one}{\mathbf{1}}
\newcommand{\zero}{\mathbf{0}}

\newcommand{\sPi}{\mathsf{\Pi}}
\newcommand{\sSigma}{\mathsf{\Sigma}}
\newcommand{\comp}{\textsc{comp}}
\newcommand{\cxt}{\textsf{cxt}}

\newcommand{\elim}{\textsc{elim}}
\newcommand{\Exch}{\mathsf{Exch}}
\newcommand{\form}{\textsc{form}}
\newcommand{\Id}{\mathsf{Id}}

\newcommand{\intro}{\textsc{intro}}
\newcommand{\oftype}{\! : \!}
\newcommand{\Subst}{\mathsf{Subst}}

\newcommand{\type}{\mathsf{type}}

\newcommand{\Weak}{\mathsf{Wkg}}
\newcommand{\Vble}{\mathsf{Vble}}

\newcommand{\types}{\vdash}

\newcommand{\lscott}{[\![}
\newcommand{\rscott}{]\!]}

\newcommand{\Ho}{\operatorname{Ho}}
\newcommand{\Hom}{\operatorname{Hom}}

\begin{document}

\title{Homotopy Theoretic Models of Type Theory}

\author{Peter Arndt\inst{1} \and Krzysztof Kapulkin\inst{2}}

\institute{University of Oslo, Oslo, Norway \email{peter.arndt@mathematik.uni-regensburg.de}
\and
University of Pittsburgh,
Pittsburgh, PA, USA \email{krk56@pitt.edu} }

%\tableofcontents

\maketitle

\begin{abstract}
  We introduce the notion of a logical model category which is a Quillen model category satisfying some additional conditions. Those conditions provide enough expressive power that one can soundly interpret dependent products and sums in it while also having a purely intensional interpretation of the identity types. On the other hand, those conditions are easy to check and provide a wide class of models that are examined in the paper.
\end{abstract}

\section{Introduction}

Starting with the Hofmann--Streicher groupoid model \cite{hofmann-streicher} it has become clear that there exist deep connections between Martin-L\"of Intentional Type Theory (see \cite{martin-lof:introduction, nordstrom:book}) and homotopy theory. Recently, these connections have been very intensively studied. We start by briefly summarizing this work---a more complete survey can be found in \cite{awodey:survey}.

It is well-known that {\em Identity Types} (or {\em Equality Types}) play an important role in type theory since they provide a relaxed and coarser notion of equality between terms of a given type. For example, assuming standard rules for type $\Nat$ one cannot prove that

$$n \oftype \Nat \types n\mbox{+}0 = n : \Nat$$

\noindent but there is still an inhabitant

$$n \oftype \Nat \types p : \Id_{\Nat}(n\mbox{+}0 \, ,\, n).$$

Identity types can be axiomatized in a usual way (as inductive types) by $\form$, $\intro$, $\elim$, and $\comp$ rules (see eg. \cite{nordstrom:book}). A type theory where we do not impose any further rules on the identity types is called {\em intensional}. One may be interested in adding the following {\em reflection rule}:

$$\inferrule*[right=$\Id$-refl]{\Gamma \types p : \Id_A(a,b)}{\Gamma \types a=b : A}$$

Now, the $\Id$ would not be any coarser than the usual definitional equality. However, the reflection rule destroys decidability of type-checking, an important property of type theory.

In order to interpret Martin-L\"of Type Theory in the categorical setting, one has to have some notion of `dependent products' and `dependent sums'. As it was shown by Seely \cite{seely:lccc}, locally cartesian closed categories (recall that $\C$ is a locally cartesian closed category if every slice of $\C$ is cartesian closed) provide a natural setting to interpret such operations. However, this interpretation forces the reflection rule to hold, that is if $p : \Id_A \to A \times A$ is an interpretation of $\Id$-type over $A$, it is automatically isomorphic to the diagonal map $\Delta \colon A \to A \times A$ in $\C/(A \times A)$.

For a semantics that does not force the reflection rule, one can pass to {\em Quillen model categories}. Model categories, introduced by Daniel Quillen (cf. \cite{quillen:book}) give an abstract framework for homotopy theory which has found many applications, for example in algebraic topoogy and algebraic geometry. The idea of interpreting type theory in model categories has been recently very intensively explored. In \cite{awodey-warren, warren:thesis} Awodey and Warren showed that the $\Id$-types can be purely intensionally interpreted as fibrant objects in a model category satisfying some additional conditions. Following this idea Gambino and Garner in \cite{gambino-garner} defined a weak factorization system in the classifying category of a given type theory.

Another direction is to build higher categorical structures out of type theory. An $\infty$-category has, apart from objects and morphisms, also $2$-morphisms between morphisms, $3$-morphisms between $2$-morphisms, and so on. All this data is organized with various kinds of composition. The theory of higher-dimensional categories has been successfully studied by many authors (see for example \cite{batanin:natural-environment, leinster:book, lurie:HTT}) and subsequently been brought into type theory by Garner and van den Berg \cite{benno-richard}, Lumsdaine \cite{lumsdaine:extended, lumsdaine:thesis}, and Garner \cite{garner:2-d-models}.

In this paper we make an attempt to obtain sound models of type theory with the type constructors $\sPi$ and $\sSigma$ within the model-categorical framework. In good cases that is when some additional coherence conditions (see \cite{benno-richard2}) are satisfied, our notion of a model extends the well-known models for the $\Id$-types. Following \cite{kapulkin:thesis} we propose a set of conditions on a model category that provide enough structure in order to interpret those type constructors. Such a model category will be called a {\em logical model category}. Our intention was to give conditions that on one hand will be easy to check but on the other hand, will provide a wide class of examples. It is important to stress that this paper presents only a part of the ongoing project \cite{arndt-kapulkin} devoted to study of $\sPi$- and $\sSigma$-types in homotopy-theoretic models of type theory. Further directions of this project may be found in the last section.

This paper is organized as follows: Sections \ref{background_tt} and \ref{background_model_cats} provide a background on type theory and abstract homotopy theory, respectively. In Section \ref{section_main_thm} we define the notion of a logical model category and show that such a category admits a sound model of a type theory with $\sPi$- and $\sSigma$-types. Next, within this section we give a range of examples of logical model categories. Finally, in Section \ref{section_future} we sketch the directions of our further research in this project.

\textbf{Acknowledgements.} We are very grateful to Thorsten Altenkirch, Steve Awodey, Richard Garner, Martin Hyland, Peter Lumsdaine, Markus Spitzweck, Thomas Streicher, and Marek Zawadowski for many fruitful and interesting conversations. The first-named author would like to thank the Topology group at the University of Oslo for their warm hospitality which he enjoyed during the preparation of this work and in particular John Rognes for arranging financial support via the Norwegian Topology Project RCN 185335/V30. The second-named author would like to acknowledge the support of the Department of Mathematics at the University of Pittsburgh (especially Prof. Paul Gartside) and the A\&S fellowship he was enjoying in the Spring Semester 2011 as well as the Department of Philosophy at Carnegie Mellon University and in particular his advisor, Steve Awodey. He dedicates this work to his Mother whose help and support for him when writing the paper went far beyond the obvious.

%%%%%%%%%%%%%%%%%%%%%%%%%%%%%%%%%%%%%%%%%%%%%%%%%%%%%%%%%%%%%%%%%
%\input background_TT.tex 2cf7c3351

\section{Background on Type Theory} \label{background_tt}

\subsection{Logical Framework of Martin-L\"of Type Theory}

In this section we will review some basic notions of type theory (cf. \cite{nordstrom:book}).

Martin-L\"of Type Theory is a dependent type theory i.e. apart from simple types and their terms it allows type dependency as in the judgement

$$\Gamma, x \oftype A \types B(x) \ \type.$$

In this example $B$ can be regarded as a family of types indexed over $A$.

There are some basic types as for example: $\mathsf{0}$, $\mathsf{1}$, $\Nat$ and some type-forming operations. The latter can be divided into two parts:
\begin{itemize}
 \item simple type-forming operations such as $A \to B$, $A \times B$, and $A + B$.
 \item operations on dependent types such as $\sPi_{x : A} B(x)$, $\sSigma_{x : A} B(x)$, and $\Id_A (x, \, y)$.
\end{itemize}

The language of type theory consists of {\em hypothetical judgements} (or just judgements) of the following four forms:

\begin{enumerate}
  \item $\Gamma \types A \ \type$
  \item $\Gamma \types A=B \ \type$
  \item $\Gamma \types a : A$
  \item $\Gamma \types a=b : A$
\end{enumerate}

There are two more forms, derivable from the ones given above:
\begin{enumerate}
  \item[5.] $\Gamma \types \Delta \ \cxt$
  \item[6.] $\Gamma \types \Delta = \Phi \ \cxt$
\end{enumerate}

Judgements of the form 5. establish that $\Delta$ is a well-formed context (in the context $\Gamma$). $\Delta$ is said to be a well-formed context if $\Delta$ is a (possibly empty) sequence of the form $(x_0 \oftype A_0, x_1 \oftype A_1 (x), \ldots, x_n \oftype A_n (x_0, \ldots, x_{n-1}))$ and

$$\Gamma \types A_0 \ \type$$

and for $i = 1, 2, \ldots n$:

$$\Gamma, x_0 \oftype A_0, x_1 \oftype A_1, \ldots, x_{i-1} \oftype A_{i-1} \types A_i (x_0, x_1, \ldots, x_{i-1})\ \type.$$

The deduction rules of Martin-L\"of Type Theory can be divided into two parts:
\begin{itemize}
 \item {\em structural} rules.
 \item rules governing the forms of types.
\end{itemize}

The structural rules are standard and may be found in the Appendix \ref{appendix_strux_rules}. The rules governing the forms of types are presented in the next section.

\subsection{Type Constructors $\sPi$ and $\sSigma$}

Given a new basic type or type former in Martin-L\"of Type Theory we need to specify:

\begin{itemize}
 \item a {\em formation} rule, providing the conditions under which we can form a certain type.
 \item {\em introduction} rules, giving the canonical elements of a type. The set of introduction rules can be empty.
 \item an {\em elimination} rule, explaining how the terms of a type can be used in derivations.
 \item {\em computation} rules, reassuring that the introduction and elimination rules are compatible in some suitable sense. Each of the computation rules corresponds to some introduction rule.
\end{itemize}

In this paper we will be interested only in two dependent type constructors: $\sPi$ and $\sSigma$ and so we will restrict our presentation to them. We should recall that under Curry-Howard isomorphism (see for example \cite{urzyczyn:book}) they correspond to the universal and existential quantification respectively (i.e. $\forall$ and $\exists$).

\paragraph{$\sPi$-types.} The version presented below may be different from some other presentations. As the elimination rule we take a weak version that is sometimes called $\sPi$-application rule.

\fbox{
\begin{minipage}[t]{0.9\textwidth}
$$\inferrule*[right=$\sPi$-\form]{\Gamma,\ x \oftype A \types B(x)\ \type}{\Gamma \types \sPi_{x : A} B(x)\ \type}$$

$$\inferrule*[right=$\sPi$-\intro]{\Gamma,\ x \oftype A \types B(x)\ \type \\ \Gamma,\ x \oftype A \types b(x) : B(x)}{\Gamma \types \lambda x \oftype A.b(x) : \sPi_{x : A} B(x)}$$

$$\inferrule*[right=$\sPi$-\elim]{\Gamma \types f \oftype \sPi_{x : A} B(x) \\ \Gamma \types a : A}{\Gamma \types \textsf{app} (f, a) : B(a)}$$

$$\inferrule*[right=$\sPi$-\comp]{\Gamma,\ x \oftype A \types B(x)\ \type \\ \Gamma,\ x \oftype A \types b(x) : B(x) \\ \Gamma \types a : A}{\Gamma \types \textsf{app}(\lambda x \oftype A .b(x), a)=b(a) : B(a)}$$

\end{minipage}}

\paragraph{$\sSigma$-types.} We use an axiomatization of the $\sSigma$-types as an inductive type from the Calculus of Inductive Construction.

\fbox{
\begin{minipage}[t]{0.9\textwidth}
$$\inferrule*[right=$\sSigma$-\form]{\Gamma \types A\ \type \\ \Gamma,\ x \oftype A \types B(x)\ \type}{\Gamma \types \sSigma_{x : A} B(x)\ \type}$$

$$\inferrule*[right=$\sSigma$-\intro]{\Gamma \types A\ \type \\ \Gamma,\ x \oftype A \types B(x)\ \type}{\Gamma,\ x \oftype A,\ y \oftype B(x) \types \textsf{pair} (x, y) : \sSigma_{x : A} B(x)}$$

$$\inferrule*[right=$\sSigma$-\elim]{\Gamma,\ z \oftype \sSigma_{x : A} B(x) \types C(z)\ \type \\ \Gamma,\ x \oftype A,\ y \oftype B(x) \types d(x,y) : C(\textsf{pair}(x, y))}{\Gamma,\ z \oftype \sSigma_{x : A} B(x) \types \textsf{split}_d (z) : C(z)}$$

$$\inferrule*[right=$\sSigma$-\comp]{\Gamma,\ z \oftype \sSigma_{x : A} B(x) \types C(z)\ \type \\ \Gamma,\ x \oftype A,\ y \oftype B(x) \types d(x,y) : C(\textsf{pair}(x, y))}{\Gamma,\ x \oftype A,\ y :\oftype B(x) \types \textsf{split}_d (\textsf{pair} (x,y))=d(x,y) : C(\textsf{pair} (x, y))}$$
\end{minipage}}

%%%%%%%%%%%%%%%%%%%%%%%%%%%%%%%%%%%%%%%%%%%%%%%%%%%%%%%%%%%%%%%%%%%%
%\input background_hom_theory.tex

\section{Background on Model Categories} \label{background_model_cats}

In this section we will gather some notions and results from model category theory.

\begin{definition}
 \textnormal{Let $\C$ be a category. We say that $f \colon A \to B$ has the {\em left lifting property} with respect to $g \colon C \to D$ or equivalently that $g$ has the {\em right lifting property} with respect to $f$ (we write $f \pitchfork g$) if every commutative square $g \circ u = v \circ f$ as below has a {\em diagonal filler} i.e. a map $j \colon B \to C$ making the diagram
  $$\xymatrix@C=2.5em{
        A \ar[r]^{u} \ar[d]_{f} & C \ar[d]^{g} \\
        B \ar[r]_{v} \ar@{..>}[ru]|-{j} &  D
        }$$
  commute (i.e. $jf = u$ and $gj = v$).}
\end{definition}

For a collection $\mathcal{M}$ of morphisms in $\C$ we denote by ${}^{\pitchfork}\mathcal{M}$ (resp. $\mathcal{M}^{\pitchfork}$) the collection of maps having the left (resp. right) lifting property with respect to all maps in $\mathcal{M}$. %Similarly $\mathcal{M}^{\pitchfork}$ denotes the collection of maps having the right lifting property with respect to all maps in $\mathcal{M}$.

\begin{definition}\label{def:wfs}
 \textnormal{A {\em weak factorization system} $(\mathcal{L}, \mathcal{R})$ on a category $\C$ consists of two collections of morphisms $\mathcal{L}$ ({\em left maps}) and $\mathcal{R}$ ({\em right maps}) %closed under composition --- follows from item 2 below
in the category $\C$ such that:
 \begin{enumerate}
  \item\label{def:wfs-1} Every map $f \colon A \to B$ admits a factorization % $f=p \circ i$
         $$\xymatrix@C=2.5em{
        A \ar[rr]^{f} \ar[rd]_{i} & & B \\
        & C \ar[ru]_{p} &
        }$$
 where $i \in \mathcal{L}$ and $p \in \mathcal{R}$.
  \item\label{def:wfs-2} $\mathcal{L}^{\pitchfork} = \mathcal{R}$ and $\mathcal{L}= {}^{\pitchfork}\mathcal{R}$.
 \end{enumerate}}
\end{definition}

\begin{examples}\label{ex:wfs} \textnormal{The following are examples of weak factorization systems:
 \begin{enumerate}
  \item There is a weak factorization system in the category $\Sets$ with: $\mathcal{L}:=$ monomor-phisms and $\mathcal{R}:= \textrm{epimorphisms}$. Note that the factorization and filling are not unique.
  \item\label{ex:wfs-gpd} There is also a weak factorization system in the category $\Gpd$ of groupoids with: $\mathcal{L}:= \textrm{injective equivalences } \textrm{ and }\; \mathcal{R}:= \textrm{fibrations %(see \ref{def:fibration})  ... doesn't find reference
}$. Recall that a functor is called {\em injective equivalence} if it is an equivalence of categories which is injective on objects. We factor a functor $F \colon X \to Y$ in $\Gpd$ as
        $$\xymatrix@C=2.5em{
        X \ar[r] \ar[rd] & \{(x,\ y,\ f) |\ x \in \ob (X),\ y\in \ob (Y),\ f \colon Fx \to y \}\ar[d] \\
         & Y
        }$$
 \end{enumerate}}
\end{examples}

We now turn towards model categories. All results and notions given without reference can be found in \cite{hovey:book}.

\begin{definition}\label{DefModelCat}
 \textnormal{A {\em model category} is a finitely complete and cocomplete category $\C$ equipped with three subcategories: $\F$ ({\em fibrations}), $\cC$ ({\em cofibrations}), and $\cW$ ({\em weak equivalences}) satisfying the following two conditions:
 \begin{enumerate}
  \item ({\em Two-of-three}) Given a commutative triangle $$\xymatrix@C=2.5em{
        A \ar[rr]^{f} \ar[rd]_{h} & & B \ar[ld]^{g}\\
        & C &
        }$$ if any two of $f$, $g$, $h$ belong to $\cW$, then so does the third.
  \item Both $(\cC, \F \cap \cW)$ and $(\cC \cap \cW, \F)$ are weak factorization systems.
 \end{enumerate}}
\end{definition}

We will refer to model categories sometimes by the tuple $(\C,\cW,\cC,\F)$ or, if no ambiguity arises, just by the underlying category $\C$. We should also mention that some authors add an additional axiom that the classes $\cW$, $\cC$, and $\F$ are closed under retracts. This in fact is redundant as it follows from Definition \ref{DefModelCat} as explained in \cite{riehl:wfs}.

From a model category $\C$ one has a functor into its associated {\em homotopy category} $\Ho(\C)$, which is the initial functor mapping the weak equivalences to isomorphisms (which defines $\Ho(\C)$ up to equivalence). %\begin{definition}
 A morphism which is both a fibration and a weak equivalence is called an {\em trivial fibration}. Similarly, a morphism which is both a cofibration and a weak equivalence is called a {\em trivial cofibration}.
%\end{definition}

\begin{examples} \textnormal{The following are examples of model categories:
 \begin{enumerate}
\item On any complete and cocomplete category $\C$ one has the {\em discrete model structure} with $\cC:=\F:=\mor \, \C$ and $\cW:=\iso \, \C$. This is the only model structure with $\cW=\iso \, \C$.
%  \item On any category $\C$ giving a model structure with $\cW=\mor\,\C$ (an {\em indiscrete model structure}) is equivalent to just giving a factorization system, since the classes of fibrations and trivial fibrations (resp. cofibrations and trivial cofibrations) coincide.
%  \item On any Grothendieck topos one has two indiscrete model structures: One in which cofibrations are the monomorphisms and fibrations given by the right lifting property and one in which cofibrations are the complemented monomorphisms and fibrations are the split epimorphisms. % better to put this in later?
%  \item There are two model structures on the category $\Top$ of topological spaces. In the first we set: $$\F = \textrm{{\em Serre fibrations} and } \cW = \textrm{weak equivalences},$$ and define $\cC = \; ^{\pitchfork}(\F \cap \cW)$ . In the second we have: $$\F = \textrm{{\em Hurewicz fibrations}, } \cC = \textrm{cofibrations, and } \cW = \textrm{weak equivalences}.$$
  \item The category $\Gpd$ of groupoids has a structure of a model category with:
  $\F := \textrm{fibrations}$, $\cC := \textrm{functors injective on objects}$ and $\cW := \textrm{categorical equivalences}$
  \item The category $\sSets:=\Sets^{\Delta^{\op}}$ of simplicial sets (where $\Delta$ is the category of finite non-empty linearly ordered sets) has a standard model structure with $\cW:=\{\textrm{those morphisms inducing isomorphisms on all homotopy groups}\}$, \newline $\cC:=\{\textrm{monomorphisms}\}$ and $\F:=(\cW \cap \cC)^\pitchfork$
%  \item For a small category $\mathbf{D}$ the category $\sSets^\mathbf{D}$ of simplicial presheaves on $\mathbf{D}$ bears a {\em global injective model structure} in which cofibrations are the monomorphisms, weak equivalences are the morphisms which are objectwise weak equivalences of simplicial sets (or, equivalently, induce isomorphisms on all homotopy group presheaves) and fibrations are defined as $(\cC \cap \cW)^\pitchfork$.
%  \item For a small site $(\mathbf{D}, \tau)$ there is a {\em local injective model structure} on simplicial presheaves on $\mathbf{D}$ with cofibrations the monomorphisms, weak equivalences those morphisms which induce isomorphisms on all homotopy group {\em sheaves}, fibrations again are defined via the lifting property. With the same definitions applied to the category of sheaves one gets the {\em injective model structure} on simplicial sheaves.
 \end{enumerate}}
\end{examples}

\begin{definition}
 \textnormal{An object $A$ is called {\em fibrant} if the canonical map $A \to \one$ is a fibration. Similarly, an object is called {\em cofibrant} if $\zero \to A$ is a cofibration.} %\textnormal{ A {\em fibrant replacement} of an object $A$ is a weak equivalence from $A$ to a fibrant object, likewise a {\em cofibrant replacement} is a weak equivalence from a cofibrant object to $A$. The codomain (resp. domain) of these weak equivalences are also referred to as (co)fibrant replacement $A$.}
\end{definition}

% \begin{remark}
% Note that, via the factorization systems, (co)fibrant replacements always exist. Often the factorization can be carried out functorially, in this case one also gets (co)fibrant replacement functors.
% \end{remark}

% \begin{definition}\label{der:path-object}
%  Let $A$ be an object in a model category $\C$. A ({\em very good}) {\em path object} $A^I$ for $A$ consists of a factorization
%  $$\xymatrix@C=2.5em{
%         A \ar[rr]^{r} \ar[rd]_{\Delta} & & A^I \ar[ld]^{p} \\
%         & A \times A &
%         }$$
%  of the diagonal map $\Delta \colon A \to A \times A$ as an trivial cofibration $r$ followed by a fibration $p$.
% \end{definition}
%
% \begin{remark}
%  Note that a path object of an object $A$ can be defined in any category with a weak factorization system $(\mathcal{L}, \mathcal{R})$ where we think of $\mathcal{L}$ as the class of trivial cofibrations and about $\mathcal{R}$ as the class of fibrations.
% \end{remark}

\begin{defprop}\label{Quillen-adjunction}
The following are equivalent for a pair of adjoint functors $L \colon \C \leftrightarrows \D : R$ between model categories:
\begin{enumerate}
 \item $L$ preserves cofibrations and trivial cofibrations.
 \item $R$ preserves fibrations and trivial fibrations.
\end{enumerate}
\textnormal{An adjoint pair satisfying these conditions is called a {\em Quillen adjunction} and it induces an adjunction $\Ho(\C) \leftrightarrows \Ho(\D)$ between the homotopy categories. It is called a {\em Quillen equivalence} if this induced adjunction is an equivalence of categories.}
\end{defprop}

{\bf Notation:} Assume $\C$ is a finitely complete category and $f \colon B \to A$ is a morphism in $\C$. The functor taking an object in the slice over $A$ to its pullback along $f$ will be denoted by $f^* \colon \C/A \to \C/B$. This functor has a left adjoint denoted by $\Sigma_f$ which takes an object in the slice over $B$ and composes it with $f$. If $f^*$ also has a right adjoint, it will be denoted by $\Pi_f$.

A model category interacts with its slice categories in the following way:

\begin{proposition}\label{SliceModelCats}
 Let $\C$ be a model category and $C$ an object in $\C$. Define that a morphism $f$ in $\C/C$ is a fibration/cofibration/weak equivalence if it is a fibration/cofibration/weak equivalence in $\C$. Then $\C/C$ is a model category with the model structure described above. Furthermore for every morphism $f$  the adjunction $\Sigma_f \dashv f^* $ is a Quillen adjunction.
\end{proposition}

\begin{defprop}\label{Rezk}
 The following are equivalent \cite[Prop. 2.7]{Rezk}:
\begin{enumerate}
 \item Pullbacks of weak equivalences along fibrations are weak equivalences
 \item For every weak equivalence $f \colon X \to Y$ the induced Quillen adjunction $\C/X \leftrightarrows \C/Y$ is a Quillen equivalence
\end{enumerate}
\textnormal{If a model category satisfies these conditions it is called {\em right proper}. From the second formulation one can deduce that right properness only depends on the class of weak equivalences, not the whole model structure (see also \cite[Cor. 1.5.21]{CisinskiAsterisque}). There also is the dual notion of {\em left properness} (pushouts of weak equivalences along cofibrations are weak equivalences).}
\end{defprop}

\begin{proposition}[cf. {\cite[Rem 2.8]{Rezk}}] \label{RezkPropernessCriteria}
 The following are true for any model category $\C$
 \begin{enumerate}
  \item If all objects in $\C$ are fibrant, then $\C$ is right proper.
  \item If $\C$ is right proper, then so are its slice categories $\C/X$.
 \end{enumerate}
\end{proposition}

 A further important property of model categories is the following:

\begin{definition}
 \textnormal{Let $\lambda$ be a regular cardinal and $(\C,\cW,\cC,\F)$ a model category. It is called $\lambda${\em -combinatorial} if the underlying category is locally $\lambda$-presentable and there are sets $I$ (resp. $J$) of morphisms between $\lambda$-presentable (resp. $\lambda$-presentable and cofibrant) objects such that $I^\pitchfork=\F \cap \cW$ and $J^\pitchfork = \F$. A model category is {\em combinatorial} if it is $\lambda$-combinatorial for some $\lambda$.}
\end{definition}

\begin{examples}\label{CombinatorialCatExamples}
\begin{enumerate}
 \item \textnormal{The category of sets with the discrete model structure is seen to be combinatorial taking $I:=\{\{0,1\}\rightarrow \{0\}, \emptyset \rightarrow \{0\}\}$ and $J:=\emptyset$.}

 \item \textnormal{The category $\sSets$ with the standard model structure is combinatorial. Being a topos it is locally presentable, with the representable functors playing the role of generators of the category. Since the initial object is the constant functor with value the empty set and cofibrations are the monomorphisms, every object is cofibrant. One can take $I:=\{\partial\Delta^n \rightarrow \Delta^n\}$ (the inclusions of the borders into full $n$-simplices) and $J:=\{ \Lambda^n_k \rightarrow \Delta^n \}$ (the inclusions of the $k$-th $n$-horns into $n$-simplices).}

% \item A prime example is the category of simplicial presheaves on a site $\mathbf{D}$, with {\em global injective model structure}. Clearly this is locally presentable; the representable functors are $\lambda$-presentable objects for some $\lambda > \#(\mathbf{D})$ and generate the category. Since the initial object is the constant sheaf with value the empty set and cofibrations are the monomorphisms every object is cofibrant. One can take $$I:=\{i \times 1 | \ \partial \Delta^n \times \Hom(-,\ X) \to \Delta^n \times \Hom(-,\ X) \}$$ where $i$ denotes the inclusion of the border into the full $n$-simplex, both understood as constant functors on $\mathbf{D}$, and $$J:=\{j \times 1 | \ \Lambda^n_k \times \Hom(-, \ X) \to \Delta^n \times \Hom(-, \ X) \}$$ with $j$ the inclusion of the $k$-th $n$-horn into the $n$-simplex.
\end{enumerate}
\end{examples}

Combinatorial model categories serve as input for two constructions whose output are further combinatorial model categories:

\begin{theorem}[cf. {\cite[Sect~A.2.8]{lurie:HTT}}] \label{functorcat}
 Let $(\C,\cW,\cC,\F)$ be a combinatorial model category and $\D$ a small category. If on the functor category $\C^\D$ one defines the following classes of morphisms,

\begin{center}
 $\cC_{inj}:=\{\text{morphisms which are objectwise in } \cC\}$

$\cW_{\C^\D}:=\{\text{morphisms which are objectwise in }\cW\}$

$\F_{proj}:=\{\text{morphisms which are objectwise in }\cF\}$

\end{center}
then one has:
\begin{itemize}
 \item $(\C^\D, \cW_{\C^\D}, \cC_{inj}, (\cW_{\C^\D} \cap \cC_{inj})^\pitchfork)$ and $(\C^\D, \cW_{\C^\D}, {}^\pitchfork(\cW_{\C^\D} \cap \F_{proj}), \F_{proj})$ are combinatorial model category structures on $\C^\D$, called the injective and the projective structure, respectively.
 \item If $(\C,\cW,\cC,\F)$ is left or right proper, then so are the above model structures.
\end{itemize}
\end{theorem}

\begin{examples}\label{functorcatExamples}
 \textnormal{\begin{enumerate}
 \item The construction applied to the discrete model structure on the category of sets yields the discrete model structure on presheaves.
 \item For a small category $\D$ the injective model structure on simplicial presheaves is an example for the above construction applied to the combinatorial left and right proper model category of simplicial sets, yielding the so-called {\em global injective model structure} on $\sSets^\D$.
 \end{enumerate}}
\end{examples}

Next we summarize some results on ({\em left Bousfield}) {\em localizations}. This is a technique to replace a given model structure on a category by another one, enlarging the class of weak equivalences, keeping the class of cofibrations and adjusting the class of fibrations accordingly. The applicability of this technique is only ensured when the model category is either {\em cellular} (for this see \cite{hirschhorn:book}) or combinatorial.

In the following theorem we will use the {\em mapping space} $\mathbb{R}\Hom(X,Y) \in \Ho(\sSets)$ which one can associate to any two objects $X,Y$ of a model category as in \cite[Sect.~5.4]{hovey:book}---for simplicial model categories this can be taken to be the simplicial hom-set of morphisms between a cofibrant replacement of $X$ and a fibrant replacement of $Y$.

\begin{theorem}[J. Smith, proven in \cite{Barwick}]
Let $(\C,\cW,\cC, \F)$ be a left proper combinatorial model category and $H$ a set of morphisms of $\Ho(\C)$. Define an object $X \in \mathcal{M}$ to be {\em $H$-local} if any morphism $f \colon A \to B$ in $H$ induces an isomorphism $f^* \colon \mathbb{R}\Hom(B,X)\to \mathbb{R}\Hom(A,X)$ in $\Ho(\sSets)$. Define $\cW_H$, the class of $H$-equivalences, to be the class of morphisms \\ \noindent $f \colon A \to B$ which induce isomorphisms  $f^*\colon \mathbb{R}\Hom(B,X)\to \mathbb{R}\Hom(A,X)$ in $\Ho(\sSets)$ for all $H$-local objects $X$. Then $(\C, \cW_H, \cC, \F_H:=(\cW_H \cap \cC)^\pitchfork)$ is a left proper combinatorial model structure.
\end{theorem}

\begin{remark}
\begin{enumerate}
 \item The fibrant objects in the localized model structure are exactly the $H$-local objects which are fibrant in the original model structure.
 \item A shorter proof of the above theorem for the special case of simplicial sheaves can be found in \cite[Sect.~2.2]{MorelVoevodsky}.
\end{enumerate}
\end{remark}

\begin{example}\label{LocalizationExamples}
% \begin{enumerate}
% \item One can obtain the usual model structure on simplicial sets by starting out with the ``discrete'' model structure having cofibrations the monomorphisms and weak equivalences the isomorphisms and localizing by the one-element class $\{\Delta^1 \to \Delta^0\}$

% \item The $\mathbb{A}^1$-homotopy category from motivic homotopy theory is constructed by localizing the standard local model structure on the category of simplicial sheaves on the site $\mathbf{Sm}/k$ of smooth schemes over $k$ with Nisnevich topology with the class $$\{ ! \times 1: \mathbb{A}^1 \times \Hom(-, \ X) \to \Hom(-, \ X) | X \in \mathbf{Sm}/k \}$$

% \item
 One can localize the discrete model structure on presheaves on a site taking $H$ to be the set of morphisms of the following form: For each cover $\{A_i\rightarrow X | i \in I\}$ in the given Grothendieck topology take the canonical morphism $\Coeq (\coprod_{I \times I}\Hom(-,A_i \times_X A_j) \rightrightarrows \coprod_I \Hom(-,A_i)) \rightarrow X.$ This yields a non-discrete model structure whose homotopy category is equivalent to the category of sheaves. Being $H$-local means in this case satisfying the descent condition for the covers given from the Grothendieck topology in question.
%\end{enumerate}
\end{example}

We still record a property of Bousfield localizations of simplicial sheaf categories which will be of interest:

\begin{theorem}[see {\cite[Thm 2.2.7]{MorelVoevodsky}}] \label{MVrightProperness}
 The Bousfield localization with respect to a set $A$ of morphisms of a category of simplicial sheaves with local model structure is right proper if there exists a set $\tilde{A}$ of monomorphisms such that

\begin{enumerate}
 \item Every arrow from $A$ is isomorphic to one from $\tilde{A}$ in the homotopy category
 \item Given an object $X$, a morphism $f:Y \to Z \in \tilde{A}$ and a fibration $p: E \to X \times Z$, the projection from the pullback $E \times _{X \times Z} X \times Y$ is a local weak equivalence.
\end{enumerate}

\end{theorem}

Last, we gather some basic facts about {\em Cisinski model structures}, a class of model structures on toposes under which many of our examples fall and which has a big overlap, but does not coincide with, the model structures which can be constructed in the ways sketched above. In the following we will use the terminology of small and large sets, small meaning to be contained in a Grothendieck universe.

\begin{definition}
Let $\C$ be a topos. A set $\cW$ of morphisms of $\C$ is called a {\em localizer} if
\begin{enumerate}
 \item $\cW$ has the two-of-three property (see \ref{DefModelCat}.1)
 \item $\cW$ contains $(Mono\, \C)^\pitchfork$
 \item $\cW$ is closed under pushouts and transfinite compositions (i.e. for a chain of morphisms in $\cW$ the canonical morphism from the domain of the first one to the colimit of the chain is again in $\cW$)
\end{enumerate}

For any set of morphisms $S$ there is a smallest localizer $\cW(S)$ containing it, namely the intersection of all localizers containing $S$. A localizer is called {\em accessible} if it is generated by a small set.
\end{definition}

\begin{theorem}[see {\cite[Thm~3.9]{CisinskiTopoi}}]
 For any accessible localizer $\cW$ in a topos $\C$, the tuple $(\C, \cW, Mono, (\cW \cap Mono)^\pitchfork)$ is a model structure.
\end{theorem}

A model structure arising in the above way is called a {\em Cisinski model structure}. Since the cofibrations are the monomorphisms and every morphism with domain an initial object is a monomorphism, every object in a Cisinski model category is cofibrant. Hence any such model structure is left proper (\ref{RezkPropernessCriteria}.1). Right properness is adressed in the following statements:

\begin{theorem}[cf. {\cite[Thm. 4.8]{CisinskiTopoi}}]\label{CisinskiGeneralProperness}
 Let $\C$ be a topos and $S$ a small set of morphisms. Then $(\C, \cW(S), Mono, (\cW \cap Mono)^\pitchfork )$ is right proper if and only if for every $f:X \rightarrow Y\, \in S$ and every fibration $p:E\rightarrow B$ with fibrant domain $E$ and every morphism $g:Y \rightarrow B$, the morphism $X\times_B E \rightarrow Y \times_B E$ (the pullback of $f$ along $p$) is in $\cW(S)$.
\end{theorem}

\begin{proposition}[cf. {\cite[Prop.~3.12 and Cor.~4.11]{CisinskiTopoi}}] \label{CisinskiSpecialProperness}
 Let $\C$ be a topos and $(X_i|i \in I)$ a small family of objects of $\C$. Then the localizer generated by the projections $\{Z \times X_i \rightarrow Z | Z \in \ob(\C)\}$ is accessible and the corresponding model structure is right proper.
\end{proposition}

\begin{example}
 An example from mathematical practice of this last kind of model structure, obtained by ``contracting'' a family of objects, is the category $\Sets^{\Delta^{op} \times \mathbf{Sm}/S}$ of $\sSets$-valued functors on smooth schemes over a base $S$ where one localizes the local injective model structure on $\sSets^{\mathbf{Sm}/S}$ by the set $$\{ ! \times 1: \mathbb{A}^1 \times \Hom(-, \ X) \to \Hom(-, \ X) | X \in \mathbf{Sm}/S \}$$
\end{example}

%%%%%%%%%%%%%%%%%%%%%%%%%%%%%%%%%%%%%%%%%%%%%%%%%%%%%%%%%%%%%%%%%
%\input main_theorem.tex
\section{Main Theorem and Examples} \label{section_main_thm}

In this section we will define the notion of a logical model category and show how one can interpret $\sPi$- and $\sSigma$-types in such a category.

\begin{definition}\label{DefLogModCat}
 We say that $\C$ is a {\em logical model category} if $\C$ is a model category and in addition the following two conditions hold:
\begin{enumerate}
 \item if $f \colon B \to A$ is a fibration in $\C$, then there exists the right adjoint $\Pi_f$ to the pullback functor $f^*$.
 \item the class of trivial cofibrations is closed under pullback along a fibration.
\end{enumerate}
\end{definition}

Clearly, one has the following corollary which provides a convenient way of checking that a model category is in fact a logical model category.

\begin{corollary}\label{LogicalModelCat} If $\C$ is a model model satisfying the following three conditions:
\begin{enumerate}
 \item \label{ML-int-Pi}if $f \colon B \to A$ is a fibration in $\C$, then there exists the right adjoint $\Pi_f$ to the pullback functor $f^*$.
 \item \label{ML-int-cof} the class of cofibrations is closed under pullback along a fibration.
 \item \label{ML-int-we} $\C$ is right proper.
\end{enumerate}
then $\C$ is a logical model category.
\end{corollary}

%----------------------------old version-------------------------------
% We say that $\C$ is a {\em logical model category} $\C$ if $\C$ is a model category and in addition the following three conditions hold:
% \begin{enumerate}
%  \item \label{ML-int-choice} $\C$ has a functorial choice of path objects (see Definition \ref{der:path-object}) in $\C$ and all of its slices which is stable under pullback i.e. let $f \colon B \to A$ be a fibration and $g \colon X \to A$ any map in $\C$, then \begin{equation}\label{eq:mod-cat-pull-stab}f^*(B^I) \cong (f^*B)^I.\end{equation}
%  \item \label{ML-int-Pi}if $f \colon B \to A$ is a fibration in $\C$, then there exists the right adjoint $\Pi_f$ to the pullback functor $f^*$.
%  \item \label{ML-int-cof:we} classes of cofibrations and weak equivalences are closed under pullback.
% \end{enumerate}

Given a logical model category $\C$ one can informally describe the interpretation of the syntax of type theory with $\sPi$- and $\sSigma$-types as follows (it can be made formal using Pitts'es formalism of {\em type categories} as described in \cite{pitts:catlog}):

\begin{itemize}
 \item contexts are interpreted as fibrant objects. In particular the empty context is interpreted as a terminal object in $\C$.
 \item a judgement $\Gamma \types A\ \type$ is interpreted as a fibration $\lscott \Gamma,\ x : A \rscott \to \lscott \Gamma \rscott$.
 \item a judgement $\Gamma \types a : A$ is interpreted as a section of $\lscott \Gamma,\ x : A \rscott \to \lscott \Gamma \rscott$ i.e.
 $$\xymatrix@C=2.5em{
        \lscott \Gamma \rscott \ar[rr]^{\lscott a \rscott} \ar@{=}[rd] & & \lscott \Gamma,\ x \colon A \rscott \ar[ld]^{\lscott\Gamma \types A\ \type\rscott} \\
        & \lscott \Gamma \rscott &
        }$$
 \item substitution along $f \colon \Gamma \to \Delta$ is interpreted by means of the pullback functor $f^*$.
\end{itemize}

% \begin{remark}
%  The key point of this interpretation (i.e.~homotopy theoretic) is that we allow only fibrant objects and fibrations as interpretations of contexts and
% \end{remark}

One can observe that the notion of model for type theory, which is a logical model category, is too strong. In fact, we do not need the whole model structure but only one weak factorization system. However, all the examples we have in mind are already model categories so introducing the notion of model in this way is not really a restriction from this point of view.

\begin{theorem}\label{main-theorem}
 If $\C$ is a logical model category, then the above interpretation is sound.
\end{theorem}

\begin{proof}
 The interpretation of $\sPi$- and $\sSigma$-types goes along the lines of \cite{seely:lccc}, that is they are interpreted by means of the right adjoint and the left adjoint to the pullback functor, respectively. The only new thing one has to show is that if $f \colon B \to A$ is a fibration, then $\Pi_f$ and $\Sigma_f$ preserve fibrant objects (to validate the formation rules). So it is enough to show that $\Pi_f$ and $\Sigma_f$ preserve fibrations. In case of $\Sigma_f$ it is clear since the class of fibrations is closed under composition. For $\Pi_f$ we observe that since fibrant objects in the slice $\C/B$ are precisely fibrations in $\C$ whose codomain is $B$ we may use theorem \ref{Quillen-adjunction} (by the two pullback lemma) to reduce the problem to condition 2. from the Definition \ref{DefLogModCat}.
\end{proof}

%%%%%%% General example

There is a big class of examples:

\begin{proposition}
\begin{enumerate}
 \item Any right proper Cisinski model structure admits a good interpretation of $\Pi$- and $\Sigma$-types, i.e. satisfies conditions \ref{ML-int-Pi}--\ref{ML-int-we}.
 \item Any left Bousfield localization of the category of sheaves on a site with the injective model structure is an example of this, provided that either \ref{MVrightProperness}, \ref{CisinskiGeneralProperness} or \ref{CisinskiSpecialProperness} apply.
\end{enumerate}
\end{proposition}
will
\begin{proof} We proceed by checking the conditions of Corollary \ref{LogicalModelCat}:

Condition \ref{ML-int-Pi}: Toposes are locally cartesian closed.

Condition \ref{ML-int-cof}: The cofibrations of the injective model structure are the monomorphisms and by definition they stay the same after a left Bousfield localisation. The class of monomorphisms is closed under pullback.

Condition \ref{ML-int-we}: Right properness is ensured by hypothesis in the first case and by \ref{MVrightProperness}, \ref{CisinskiGeneralProperness}, \ref{CisinskiSpecialProperness} in the second.
\end{proof}

\begin{proposition}
 If $\C$ is a logical model category, then so is any slice category $\C/X$.
\end{proposition}
\begin{proof}
 By definition of the model structure in \ref{SliceModelCats}, the cofibrations, fibrations and weak equivalences in $\C/X$ are defined to be those of $\C$. Since pullbacks in a slice category are also pullbacks in the original category, trivial cofibrations are preserved under pullback along fibrations by hypothesis. The right adjoint of pullback along a fibration in $\C/X$ is the one from $\C$, the adjointness property follows from the one for $\C$ using that $(\C/X)/A \cong \C/A$ for every object $A\rightarrow X$ of $\C/X$.
\end{proof}

%%%%%%%%%%%%%%%%%%%%%%%%%%%%%%%%%%%%%%%%%%%%%%%%%%%%%%%%%%%%%%%%%%%%%%%%%%%%%%
%\input examples.tex
We now give some concrete examples of model categories satisfying conditions (\ref{ML-int-Pi})--(\ref{ML-int-we}).

\subsection{Groupoids}

Our first example is the category $\Gpd$ of groupoids. It is well known (see \cite{giraud, conduche}) that the right adjoint $\Pi_f$ to the pullback functor exists in $\Gpd$ (as a subcategory of $\Cat$) if and only if $f$ is a so-called {\em Conduch\'e fibration}. In particular, any fibration or opfibration is a Conduch\'e fibration. Since we are interested only in taking $\Pi_f$ for $f$ being a fibration, the condition (\ref{ML-int-Pi}) is satisfied. Alternatively, as was pointed out by a referee, one can deduce satisfaction of (\ref{ML-int-Pi}) from the fact that we are only interested in the restriction of the pullback functor to fibrations and this restricted pullback functor always has a right adjoint (no matter what we are pulling back along).
Conditions (\ref{ML-int-cof}) and (\ref{ML-int-we}) also hold. It is standard to verify that a pullback of a functor injective on objects (resp. an equivalence of categories) is again injective on objects (an equivalence).

The model obtained above is known as the Hofmann-Streicher {\em groupoid model} which was the first intensional model of Martin-L\"of Type Theory (see \cite{hofmann-streicher}).

\subsection{Extreme Examples}

%While homotopically uninteresting, the following are instructive for studying the interaction of type theory and model categories:

{\bf Discrete model structures:} Any bicomplete, locally cartesian closed category endowed with the {\em discrete model structure} ($\cC:=\F:=\mor \, \C$ and $\cW:=\iso \, \C$) satisfies conditions (\ref{ML-int-Pi}), (\ref{ML-int-cof}) and (\ref{ML-int-we}): The right adjoint to the pullback functor exists by hypothesis, the cofibrations are trivially closed under pullbacks and since every object is fibrant right properness is ensured by the first criterion given in \ref{RezkPropernessCriteria}. The interpretation of type theories given in the last section then coincides with the usual extensional one.

{\bf Two indiscrete model structures on a Grothendieck topos:} On any Grothendieck topos (automatically satisfying condition (\ref{ML-int-Pi})) one has two model structures in which the weak equivalences are all morphisms: A Cisinski model structure, and one in which the cofibrations are the {\em complemented monomorphisms} and the fibrations are the split epimorphisms. Right properness of this latter model structure follows from the fact that the first model structure is right proper together with the fact that properness depends only on the weak equivalences (\ref{Rezk}). Stability of cofibrations under pullback is given by the fact that pulling back commutes with taking complements in a topos.

{\bf The minimal Cisinski model structure on $\Sets$:} There is a model structure on $\Sets$, see \cite{CisinskiTopoi} Ex. 3.7, such that cofibrations are monomorphisms and weak equivalences are all morphisms except those whose domain is the empty set and whose codomain is not. The fibrations in this model structure are the epimorphisms and the maps whose domain is the empty set. It is the minimal Cisinski model structure, $\cW := \cW(\emptyset)$ and thus by \ref{CisinskiSpecialProperness}, applied to the empty set of morphisms, is right proper.

\subsection{Localized Structures on (Pre)sheaf Categories}

Taking a category of presheaves on a site (it is locally cartesian closed, ensuring (\ref{ML-int-Pi})) we can endow it with the discrete model structure. By \ref{CombinatorialCatExamples} the discrete model structure on the category of sets is combinatorial, hence by \ref{functorcat} and \ref{functorcatExamples}(1) so is the one on presheaves. It is left and right proper since all objects are fibrant and cofibrant (by \ref{RezkPropernessCriteria}) and hence one can apply left Bousfield localization.

{\bf Localization by Grothendieck topologies:} Examples from mathematical practice are ubiquitous and include the $\Sets$-valued presheaves localized by a Grothendieck topology as in \ref{LocalizationExamples}(3)---e.g. by the Grothendieck topologies for sheaf toposes of topological spaces, for the \'etale, crystalline, or Zariski toposes of schemes, Weil-\'etale toposes or classifying toposes of geometric theories.

{\bf Presheaves on Test categories:} There is a theory of functor categories on test categories that are categories supporting model structures which are Quillen equivalent to the standard model structure on topological spaces, see \cite{CisinskiAsterisque}. The best known examples are cubical sets and simplicial sets. To give an idea about how the further studies of these examples will be pursued we add a description of the $\Pi$-functor for $\sSets$:

Let $p \colon X \to B$ be an object in $\sSets/B$ and $f \colon B \to A$ a fibration of simplicial sets. Moreover let $a \colon \Delta^n \to A$ be a cell inclusion and $f^*(a) =: \overline{a}$. Then $$\Pi_f (X, p)_n = \{ h \colon \Delta^n \times_A B \to X | \ p \circ h = \overline{a}\}.$$

{\bf Models for the $\mathbb{A}^1$-homotopy category:} Let $S$ be some base scheme and $\mathbf{Sm}/S$ the category of smooth schemes over $S$. There are two Quillen equivalent Cisinski models for the $\mathbb{A}^1$-homotopy category, one model structure on the category of $\Sets$-valued Nisnevich sheaves on $\mathbf{Sm}/S$,  and one on the category of $\sSets$-valued Nisnevich sheaves on $\mathbf{Sm}/S$. The latter is obtained by localizing the injective model structure on $\sSets^{\mathbf{Sm}/S}$ by the class $\{ ! \times 1: \mathbb{A}^1 \times \Hom(-, \ X) \to \Hom(-, \ X) | X \in \mathbf{Sm}/S \}$; the former by a similar process; see \cite{VoevodskyICM} for details. We remark that these are presentations of a locally cartesian closed $\infty$-category which is not an $\infty$-topos.

%%%%%%%%%%%%%%%%%%%%%%%%%%%%%%%%%%%%%%%%%%%%%%%%%%%%%%%%%%%%%%%%%%%%%%%%%%%%
%\input future.tex
\section{Future Research} \label{section_future}

In this section we show further directions that are of our interest in \cite{arndt-kapulkin}.

 \paragraph{Semantics in fibration categories.} One can observe that our interpretation of $\sPi$-, $\sSigma$-, and $\Id$-types uses in fact one weak factorization system, not the whole model structure. Since a weak factorization system on its own does not provide a notion of homotopy, one may wish to look for a different notion of semantics. A framework to address such a question can be provided by {\em fibration categories} (see \cite{baues, brown:AHT}). Fibration categories have a well-behaved notion of homotopy, while also seem to be rich enough to admit an interpretation of type theory.

 \paragraph{Further properties of $\sPi$-types.} One may recognize in $\sPi$-$\comp$ the standard $\beta$-rule from $\lambda$-calculus. This rule is by Theorem \ref{main-theorem} satisfied in any logical model category. In \cite{arndt-kapulkin} we will address the question which among other rules that one can associate with $\sPi$-types (such as functional extensionality or $\eta$-rule; see \cite{garner:dep-prod} for more detailed treatment) are satisfied in logical model categories.

%\paragraph{Define $\Pi$-types using $\Pi_f$ and fibrant replacement.} It is clear that the preservation of fibrations by $\Pi_f$ was the main issue in the homotopy theoretic interpretation of type theory. What one may do instead is to apply $\Pi_f$ and then take the fibrant replacement. The operation of fibrant replacement applied to an object provide an object which is fibrant and weakly equivalent to the former. Such interpretation of $\Pi$-types provides further interesting examples such as the category of symmetric spectra and projective model structures on presheaves, which enjoy a universal property.

\paragraph{Extending the class of categories admitting an interpretation of $\Pi$-types.} One may try broaden the scope of our interpretation of $\Pi$-types in two directions. First, our definition of logical model category required the existence of a right adjoint to the pullback functor along a fibration. It seems enough, however, to require a right adjoint to the restriction of this pullback functor  to the full subcategory whose objects are fibrations. Second, we required the right adjoint to take fibrations to fibrations. Alternatetively one may try to force this property using the fibrant replacement functor in a model category. However, in both cases a careful consideration of coherence issues is required, while in the setting of this work the standard methods apply to give coherence.

\appendix
\section{Structural rules of Martin-L\"of Type Theory} \label{appendix_strux_rules}

The structural rules of Martin-L\"of Type Theory are:

\fbox{
\begin{minipage}[t]{1.0\textwidth}
\noindent Variables \hspace{29mm} Substitution \vspace{2mm} \\
\vspace{3mm}
$\inferrule*[right=$\Vble$]{\Gamma \types A\ \type}{\Gamma,\ x \colon A \types x \colon A}$ \hspace{10mm}
$\inferrule*[right=$\Subst$]{\Gamma \types a \colon A \\ \Gamma,\ x \colon A,\ \Delta \types \mathcal{I}}{\Gamma,\ \Delta[a/x] \types \mathcal{I}[a/x]}$ \\
\noindent Weakening \hspace{27mm} Exchange \vspace{2mm} \\ \vspace{3mm}
$\inferrule*[right=$\Weak$]{\Gamma \types A\ \type \\ \Gamma \types \mathcal{I}}{\Gamma,\ x \colon A \types \mathcal{I}}$
%\item[]
$\ \ \inferrule*[right=$\Exch$]{\Gamma,\ x \colon A,\ y \colon B,\ \Delta \types \mathcal{I}}{\Gamma,\ y \colon B,\ x \colon A,\ \Delta \types \mathcal{I}}$ if $x$ is not free in $B$.

\noindent Definitional (Syntactic) Equality
$$\inferrule*{\Gamma \types A\ \type}{\Gamma \types A = A\ \type} \;\;\;\;\; \inferrule*{\Gamma \types A=B\ \type}{\Gamma \types B = A\ \type} \;\;\;\;\; \inferrule*{\Gamma \types A=B\ \type \\ \Gamma \types B=C \ \type}{\Gamma \types A = C\ \type}$$
$$\inferrule*{\Gamma \types a \colon A}{\Gamma \types a = a \colon A} \;\;\;\;\; \inferrule*{\Gamma \types a=b \colon A}{\Gamma \types b=a \colon A} \;\;\;\;\; \inferrule*{\Gamma \types a=b \colon A \\ \Gamma \types b=c \colon A}{\Gamma \types a=c \colon A}$$

\end{minipage}}
% \begin{itemize}
% \item[] Variables \hspace{33mm} Substitution  \\
% \vspace{3mm}
% $\inferrule*[right=$\Vble$]{\Gamma \types A\ \type}{\Gamma,\ x \colon A \types x \colon A}$ \hspace{10mm}
% $\inferrule*[right=$\Subst$]{\Gamma \types a \colon A \\ \Gamma,\ x \colon A,\ \Delta \types \mathcal{I}}{\Gamma,\ \Delta[a/x] \types \mathcal{I}[a/x]}$
% \item[] Weakening \hspace{33mm} Exchange \\ \vspace{3mm}
% $\inferrule*[right=$\Weak$]{\Gamma \types A\ \type \\ \Gamma \types \mathcal{I}}{\Gamma,\ x \colon A \types \mathcal{I}}$
% %\item[]
% $\ \ \inferrule*[right=$\Exch$]{\Gamma,\ x \colon A,\ y \colon B,\ \Delta \types \mathcal{I}}{\Gamma,\ y \colon B,\ x \colon A,\ \Delta \types \mathcal{I}}$ if $x$ is not free in $B$.
% \item[] Equality
% $$\inferrule*{\Gamma \types A\ \type}{\Gamma \types A = A\ \type} \;\;\;\;\; \inferrule*{\Gamma \types A=B\ \type}{\Gamma \types B = A\ \type} \;\;\;\;\; \inferrule*{\Gamma \types A=B\ \type \\ \Gamma \types B=C \ \type}{\Gamma \types A = C\ \type}$$
% $$\inferrule*{\Gamma \types a \colon A}{\Gamma \types a = a \colon A} \;\;\;\;\; \inferrule*{\Gamma \types a=b \colon A}{\Gamma \types b=a \colon A} \;\;\;\;\; \inferrule*{\Gamma \types a=b \colon A \\ \Gamma \types b=c \colon A}{\Gamma \types a=c \colon A}$$
% \end{itemize}

\bibliographystyle{amsalpha}
\bibliography{bibliography}

\end{document}